\documentclass[english]{extarticle}
\usepackage[T1]{fontenc}
\usepackage[latin9]{inputenc}
\usepackage{geometry}
\usepackage{enumitem}
\usepackage{amsmath}
\usepackage{amsthm}
\usepackage{amssymb}
\PassOptionsToPackage{normalem}{ulem}
\usepackage{ulem}

\makeatletter
  \theoremstyle{definition}
  \newtheorem{example}{\protect\examplename}[section]
  \theoremstyle{definition}
  \newtheorem{defn}{\protect\definitionname}[section]
  \theoremstyle{remark}
  \newtheorem*{rem}{Remark}
  \theoremstyle{plain}
  \newtheorem{thm}{\protect\theoremname}[section]
  \theoremstyle{plain}
  \newtheorem{prop}{\protect\propositionname}[section]

\DeclareMathOperator{\sgn}{sign}
\DeclareMathOperator{\diag}{diag}

\author{Lin Jiu\\
Department of Mathematics and Statictics\\
Dalhousie University\\
6316 Coburg Road\\
Halifax, Nova Scotia, Canada B3H 4R2\\
Lin.Jiu@dal.ca\\
\\
and\\
\\
Diane Yahui Shi\footnote{Corresponding author}\\
School of Mathematics\\
Tianjin University\\
Tianjin 300072, P. R. China\\
shiyahui@tju.edu.cn}

\date{}

\makeatother

\usepackage{babel}
  \providecommand{\definitionname}{Definition}
  \providecommand{\examplename}{Example}
  \providecommand{\propositionname}{Proposition}
  
\providecommand{\theoremname}{Theorem}

\begin{document}

\title{Matrix Representations for Multiplicative Nested Sums}
\maketitle
\begin{abstract}
We study multiplicative nested sums, which are generalizations
of harmonic sums, and provide a calculation through multiplication
of index matrices. Special cases interpret the index matrices as stochastic
transition matrices of random walks on a finite number of sites. Relations
among multiplicative nested sums, which are generalizations of relations
between harmonic series and multiple zeta functions, can be easily
derived from identities of the index matrices. Combinatorial identities
and their generalizations can also be derived from this computation.
\end{abstract}

2010 Classification: Primary 11C20, Secondary 05A19

\section{Introduction}

The harmonic sums, defined by \cite[eq.~4, pp.~1]{HarmonicSSumDEF}
\begin{equation}
S_{i_{1},\ldots,i_{k}}\left(N\right)=\sum_{N\geq n_{1}\geq\cdots\ge n_{k}\ge1}\frac{\sgn\left(i_{1}\right)^{n_{1}}}{n_{1}^{\left|i_{1}\right|}}\times\cdots\times\frac{\sgn\left(i_{k}\right)^{n_{k}}}{n_{k}^{\left|i_{k}\right|}},\label{eq:HarmonicSSum}
\end{equation}
and \cite[pp.~168]{HarmonicH}
\begin{equation}
H_{i_{1},\ldots,i_{k}}\left(N\right)=\sum_{N>n_{1}>\cdots>n_{k}\ge1}\frac{\sgn\left(i_{1}\right)^{n_{1}}}{n_{1}^{\left|i_{1}\right|}}\times\cdots\times\frac{\sgn\left(i_{k}\right)^{n_{k}}}{n_{k}^{\left|i_{k}\right|}},\label{eq:HarmonicHSum}
\end{equation}
are naturally connected to zeta functions. For instance, \\
1. taking $k=1$, $i_1=x>0$ and $N\rightarrow\infty$, in either \eqref{eq:HarmonicSSum} or \eqref{eq:HarmonicHSum}, gives Riemann zeta-function $\zeta(x)$; \\
2. when $i_1,\ldots,i_k>0$ and $N\rightarrow\infty$, \eqref{eq:HarmonicHSum} becomes mulitple zeta value $\zeta(i_1,\ldots,i_k)$. \\
Applications of harmonic sums appear in various areas, such as
\cite[pp.~1]{JakobThesis}  perturbative calculations of massless or massive single scale problems in quantum field theory. Ablinger \cite[Chpt.~6]{JakobThesis} implemented the \texttt{Mathematica}
package \texttt{HarmonicSums.m}\footnote{http://www.risc.jku.at/research/combinat/software/HarmonicSums/index.php},
based on the recurrence \cite[eq.~2.1, pp.~21]{HarmonicSums} that is
inherited from the quasi-shuffle relation \cite[eq.~1, pp.~51]{QuasiShuffle},
for calculation of harmonic sums.

The current work here is to present an alternative calculation for, not only harmonic sums, but also for the general sums defined as follows.

\begin{defn}
We consider the following \emph{multiplicative nested sums} (MNS):
for $m,N\in\mathbb{N}$,
\begin{equation}
S\left(f_{1},\ldots,f_{k};N,m\right):=\sum_{N\geq n_{1}\geq\cdots\geq n_{k}\geq m}f_{1}\left(n_{1}\right)\cdots f_{k}\left(n_{k}\right),\label{eq:S}
\end{equation}
and 
\begin{equation}
A\left(f_{1},\ldots,f_{k};N,m\right):=\sum_{N>n_{1}>\cdots>n_{k}\geq m}f_{1}\left(n_{1}\right)\cdots f_{k}\left(n_{k}\right).\label{eq:A}
\end{equation}
That is, the usual summand is multiplicative 
$f\left(n_{1},\ldots,n_{k}\right)=f_{1}\left(n_{1}\right)\cdots f_{k}\left(n_{k}\right)$, and the summation indices are nested. Here, for all \emph{$l=1,\ldots,k$},
$f_{l}$ can be any function defined on $\left\{ m,m+1,\ldots,N\right\} $,
unless $N=\infty$ when convergence needs to be taken into consideration. 
\end{defn}
\begin{rem}
Let $f_k(x):=\sgn(x)^{i_k}/x^{i_k}$, then, \eqref{eq:S} gives \eqref{eq:HarmonicSSum} and \eqref{eq:A} gives \eqref{eq:HarmonicHSum}
\end{rem}

In Section \ref{sec:GMR}, we present the main theorem, i.e., the calculation for MNS, by associating to each
function $f_{l}$ an \emph{index matrix} and then considering the multiplications. Since MNS are generalizations of harmonic sums, this method naturally works for harmonic sums. Rather than recursively applying the quasi-shuffle relations in \cite{JakobThesis}, this matrix calculation is more direct and also simultaneously calculate for multiple pairs of $N$ and $m$. 
Properties of index matrix, such as inverse, identities, and eigenvalues, eigenvectors,
diagonalization, follow after the main theorem. 

Applications of this matrix calculation, presented in Section \ref{sec:Applications}, connect different fields. Originally, this idea was
inspired by constructing random walks for special harmonic sums. Different
types of random walks appear in and
connect to various fields. For example, the coefficients connecting
Euler polynomials and generalized Euler polynomials \cite[eq.~3.8, pp.~781]{EulerRW}
appear in a random walk over a finite number of sites \cite[Note 4.8, pp.~787]{EulerRW}.
In Subsection \ref{subsec:RW}, the special sum when $f_{1}=\cdots=f_{k}=x^{-a}$
for $a\geq1$ is interpreted as the probability of a certain random walk, while the index matrix is exactly the corresponding stochastic
matrix.

Consider the limit case of harmonic sums \eqref{eq:HarmonicSSum} and \eqref{eq:HarmonicHSum}, as $N\rightarrow\infty$ and further assuming $i_{1},\ldots,i_{k}>0$, i.e., $S\left(1/x^{i_1},\ldots,1/x^{i_k};\infty,1\right)$ and $A\left(1/x^{i_1},\ldots,1/x^{i_k};\infty,1\right)$. The relation between them are of great importance and interest. For instance, the fact 
\[
S\left(\frac{1}{x^2},\frac{1}{x};\infty,1\right)=2A\left(\frac{1}{x^3};\infty,1\right)=2\zeta\left(3\right)
\]
has been well studied and rediscovered many times. Hoffman \cite[Thm.~2.1, pp.~277, Thm.~2.2, pp.~278]{Hoffman} obtained the symmetric sums of $S\left(1/x^{i_1},\ldots,1/x^{i_k};\infty,1\right)$ and 
$A\left(1/x^{i_1},\ldots,1/x^{i_k};\infty,1\right)$ in terms of the Riemann zeta-function $\zeta$. In particular, when $k=2$ and $k=3$,  the direct relations between
$S\left(1/x^{i_1},\ldots,1/x^{i_k};\infty,1\right)$ and $A\left(1/x^{i_1},\ldots,1/x^{i_k};\infty,1\right)$
(see \cite[pp.~276]{Hoffman} or \eqref{eq:SATwoZeta} and \eqref{eq:SAThreeZeta} below) are also easy to obtain. In Subsection \ref{subsec:SANDA},
we provide the truncated and generalized version of these relations,
easily derived from identities of the index matrices.

Finally, we focus on combinatorial identities, where harmonic sums
also appear. For instance, Dilcher \cite[Cor.~3, pp.~93]{Dilcher1}
established, for special harmonic sum, 
\begin{equation}
S_{\underset{k}{\underbrace{1,\ldots,1}}}\left(N\right)=\sum_{l=1}^{N}{N \choose l}\frac{\left(-1\right)^{l-1}}{l^{k}},\label{eq:Dilcher}
\end{equation}
from $q$-series of divisor functions. In Subsection \ref{subsec:COMB},
we present examples of combinatorial identities, including a generalization of
(\ref{eq:Dilcher}). Here, those identities are obtained by applying calculations and properties of the index matrices. Therefore, all the examples can be viewed as alternative proofs. 

\section{\label{sec:GMR}Matrix computation and properties}

\subsection{Index matrices and computations for MNS}
\begin{defn}
Given a positive integer $N$ and a function $f$ on $\left\{ 1,\ldots,N\right\} $,
we define the following $N\times N$ (lower triangular)\emph{ index
matrices}:

\begin{equation}
\mathbf{S}_{f}:=\left(\begin{array}{ccccc}
f\left(1\right) & 0 & 0 & \cdots & 0\\
f\left(2\right) & f\left(2\right) & 0 & \cdots & 0\\
\vdots & \vdots & \vdots & \ddots & \vdots\\
f\left(N\right) & f\left(N\right) & f\left(N\right) & \cdots & f\left(N\right)
\end{array}\right)
\label{eq:SIM}
\end{equation}
and 
\begin{equation}
\mathbf{A}_{f}:=\left(\begin{array}{cccccc}
0 & 0 & 0 & \cdots & 0 & 0\\
f\left(1\right) & 0 & 0 & \cdots & 0 & 0\\
f\left(2\right) & f\left(2\right) & 0 & \cdots & 0 & 0\\
\vdots & \vdots & \vdots & \ddots & \vdots & \vdots\\
f\left(N-1\right) & f\left(N-1\right) & f\left(N-1\right) & \cdots & f\left(N-1\right) & 0
\end{array}\right).
\label{eq:AIM}
\end{equation}
\end{defn}
\begin{rem}
1.  Shifting $\mathbf{S}_{f}$ downward by one row gives $\mathbf{A}_{f}$,
i.e., 
\begin{equation}
\mathbf{A}_{f}=\left(\delta_{i-1,j}\right)_{N\times N}\mathbf{S}_{f}\text{, where }\delta_{a,b}=\begin{cases}
1, & \text{if }a=b;\\
0, & \text{otherwise.}
\end{cases}\label{eq:S2A}
\end{equation}
For simplicity, we further denote $\mathbf{\Delta}:=\left(\delta_{i-1,j}\right)_{N\times N}$
so that $\mathbf{A}_{f}=\mathbf{\Delta}\mathbf{S}_{f}$. 

\noindent 2. When the dimensions of index matrices needs to be clarified, we use $\mathbf{S}_{N\mid f}$ and $\mathbf{A}_{N\mid f}$.
\end{rem}
\begin{thm}
\label{Thm:Main}Let 
\[
\mathbf{P}=\left(\begin{array}{ccccc}
1 & 0 & 0 & \cdots & 0\\
1 & 1 & 0 & \cdots & 0\\
\vdots & \vdots & \vdots & \ddots & \vdots\\
1 & 1 & 1 & \cdots & 1
\end{array}\right)_{N\times N}.
\]
Then, we have
\begin{equation}
S\left(f_{1},\ldots,f_{k};N,m\right)=\left(\mathbf{P}\cdot\prod_{l=1}^{k}\mathbf{S}_{f_{l}}\right)_{N,m},\label{eq:MatrixS}
\end{equation}
\begin{equation}
A\left(f_{1},\ldots,f_{k};N,m\right)=\left(\mathbf{P}.\prod_{l=1}^{k}\mathbf{A}_{f_{l}}\right)_{N,m},\label{eq:MatrixA}
\end{equation}
where $\mathbf{M}_{i,j}$ denotes the entry located at the $i$\textsuperscript{th}
row and $j$\textsuperscript{th} column of a matrix $\mathbf{M}$. 
\end{thm}
\begin{proof}
Since the proof for $A(f_1,\ldots,f_k;N,m)$ is similar, we shall only prove the stronger
result for $S(f_1,\ldots,f_k;N,m)$, namely,
\[
S\left(f_{1},\ldots,f_{k};i,j\right)=\left(\mathbf{P}\cdot\prod_{l=1}^{k}\mathbf{S}_{f_{l}}\right)_{i,j}.
\]

\noindent 1. When $k=1$, it is easy to see that $\left(\mathbf{P}\cdot\mathbf{S}_{f_{1}}\right)_{i,j}=\overset{i}{\underset{l=j}{\sum}}f_{1}(l)=S\left(f_{1};i,j\right).$

\noindent 2. Suppose $S\left(f_{1},\ldots,f_{k};i,j\right)=\left(\mathbf{P}\cdot\overset{k}{\underset{l=1}{\prod}}\mathbf{S}_{f_{l}}\right)_{i,j}$.
Then,
\begin{align*}
S\left(f_{1},\ldots,f_{k+1};i,j\right) & =\left(\mathbf{P}\cdot\prod_{l=1}^{k+1}\mathbf{S}_{f_{l}}\right)_{i,j}=\left(\left(\mathbf{P}\cdot\overset{k}{\underset{l=1}{\prod}}\mathbf{S}_{f_{l}}\right)\cdot\mathbf{S}_{f_{k+1}}\right)_{i,j}\\
 & =\sum_{l=j}^{i}S\left(f_{1},\ldots,f_{k};i,l\right)f_{k+1}\left(l\right)\\
 & =\sum_{l=j}^{i}f_{k+1}\left(l\right)\sum_{i\geq n_{1}\geq\cdots\geq n_{k}\geq l}f_{1}\left(n_{1}\right)\cdots f_{k}\left(n_{k}\right)\\
 & =\sum_{i\geq n_{1}\geq\cdots\geq n_{k}\geq l\geq j}f_{1}\left(n_{1}\right)\cdots f_{k}\left(n_{k}\right)f_{k+1}\left(l\right)\\
 & =S\left(f_{1},\ldots,f_{k+1};i,j\right).
\end{align*}
\end{proof}

\subsection{Properties of the index matrix $\mathbf{S}$}

To simplify expressions in this section, we denote 
\[
\mathbf{S}_{a}:=\left(\begin{array}{ccccc}
a_{1} & 0 & 0 & \cdots & 0\\
a_{2} & a_{2} & 0 & \cdots & 0\\
\vdots & \vdots & \vdots & \ddots & \vdots\\
a_{N} & a_{N} & a_{N} & \cdots & a_{N}
\end{array}\right)
\]
and $\mathbf{A}_{a}:=\mathbf{\Delta}\mathbf{S}_{a}$. In another word, we assume, in \eqref{eq:SIM} and \eqref{eq:AIM}, $f(l)=a_l$ for all $l=1,\ldots,N$, so that we replace the lower index $f$ by $a$. Next, we give
some properties of $\mathbf{S}_{a}$.
\begin{prop}
1. The inverse of $\mathbf{S}_{a}$ is given by 
\[
\mathbf{S}_{a}^{-1}=\left(\begin{array}{cccccc}
1/a_{1} & 0 & 0 & \cdots & 0 & 0\\
-1/a_{1} & 1/a_{2} & 0 & \cdots & 0 & 0\\
0 & -1/a_{2} & 1/a_{3} & \cdots & 0 & 0\\
\vdots & \vdots & \vdots & \ddots & \vdots & \vdots\\
0 & 0 & 0 & \cdots & -1/a_{N-1} & 1/a_{N}
\end{array}\right).
\]
2. We have the matrix identities 
\begin{equation}
\mathbf{S}_{a}^{-1}\mathbf{S}_{ab}\mathbf{S}_{b}^{-1}=\mathbf{I}-\mathbf{\Delta},\label{eq:SATwoMatrix}
\end{equation}
\begin{equation}
\mathbf{S}_{a}\mathbf{\Delta}\mathbf{S}_{b}\mathbf{\Delta}\mathbf{S}_{c}+\mathbf{S}_{ab}\mathbf{\Delta}\mathbf{S}_{c}+\mathbf{S}_{a}\mathbf{\Delta}\mathbf{S}_{bc}+\mathbf{S}_{abc}=\mathbf{S}_{a}\mathbf{S}_{b}\mathbf{S}_{c}.\label{eq:SAThreeMatrix}
\end{equation}
3. $\mathbf{S}_{a}$ has eigenvalues $\left\{ a_{1},\ldots,a_{N}\right\} $.
Suppose all the $a_{j}$ are distinct, then define $\mathbf{D}_{a}=\left(d_{i,j}\right)_{N\times N}$
and $\mathbf{E}_{a}:=\left(e_{i,j}\right)_{N\times N}$ by if $i\geq j$
\[
d_{i,j}:=\frac{a_{i}}{a_{N}}\prod_{k=i+1}^{N}\left(1-\frac{a_{k}}{a_{j}}\right)\text{ and }
e_{i,j}:=
\frac{a_{N}}{a_{i}}\overset{N}{\underset{\genfrac{}{}{0pt}{}{k=j}{k\neq i}}{\prod}}\frac{1}{1-\frac{a_{k}}{a_{i}}},
\]
otherwise $d_{i,j}=0=e_{i,j}$.
It follows that $\left(d_{1,j},\ldots,d_{N,j}\right)^{T}$ is an eigenvector of $\mathbf{S}_a$, 
with respect to $a_{j}$, and $\mathbf{D}_{a}^{-1}=\mathbf{E}_{a}$,
implying
\begin{equation}
\mathbf{S}_{a}=\mathbf{D}_{a}\diag\left(a_{1},\ldots,a_{N}\right)\mathbf{E}_{a},\label{eq:Diagonalization}
\end{equation}
where $\diag\left(a_{1},\ldots,a_{N}\right)$ means the diagonal matrix
with entries $\left\{ a_{1},\ldots,a_{n}\right\} $ on the diagonal.
\end{prop}
\begin{proof}
We omit the straightforward computation and only sketch the idea here.

\noindent 1. The inverse can be easily computed.

\noindent 2. Denote $\mathbf{I}_{a}:=\diag\left(1/a_1,\ldots,1/a_N\right)$
and $\mathbf{\Delta}_{a}=\mathbf{\Delta}\mathbf{I}_{a}$ so that $\mathbf{S}_{a}^{-1}=\mathbf{I}_{a}-\mathbf{\Delta}_{a}$.
Since $\mathbf{I}_{a}\mathbf{S}_{ab}=\mathbf{S}_{b}$ (but $\mathbf{S}_{ab}\mathbf{I}_{b}\neq\mathbf{S}_{a}$)
and $\mathbf{\Delta}_{a}\mathbf{S}_{ab}=\mathbf{A}_{b}=\mathbf{\Delta}\mathbf{S}_{b}$,
one easily obtains
\[
\mathbf{S}_{a}^{-1}\mathbf{S}_{ab}\mathbf{S}_{b}^{-1}=\left(\mathbf{I}_{a}-\mathbf{\Delta}_{a}\right)\mathbf{S}_{ab}\mathbf{S}_{b}^{-1}=\mathbf{S}_{b}\mathbf{S}_{b}^{-1}-\mathbf{\Delta}\mathbf{S}_{b}\mathbf{S}_{b}^{-1}=\mathbf{I}-\mathbf{\Delta}.
\]
Similarly, multiplying by $\mathbf{S}_{a}^{-1}$ from the left and by
$\mathbf{S}_{c}^{-1}$ from the right on (\ref{eq:SAThreeMatrix}),
we obtain
\[
\mathbf{\Delta}\mathbf{S}_{b}\mathbf{\Delta}+\mathbf{S}_{a}^{-1}\mathbf{S}_{ab}\mathbf{\Delta}+\mathbf{\Delta}\mathbf{S}_{bc}\mathbf{S}_{c}^{-1}+\mathbf{S}_{a}^{-1}\mathbf{S}_{abc}\mathbf{S}_{c}^{-1}=\mathbf{S}_{b},
\]
which reduces to $\mathbf{I}-\mathbf{\Delta}=\mathbf{S}_{b}^{-1}\mathbf{S}_{bc}\mathbf{S}_{c}^{-1}$,
i.e., (\ref{eq:SATwoMatrix}). 

\noindent 3. The eigenvalues are easy to see. To verify the eigenvectors,
it is equivalent to prove that for all $i=j,\ldots,N$, we have
\begin{equation}
\sum_{l=j}^{i}a_{l}\frac{a_{i}}{a_{N}}\prod_{k=l+1}^{N}\left(1-\frac{a_{k}}{a_{j}}\right)=a_{j}\frac{a_{i}}{a_{N}}\prod_{k=i+1}^{N}\left(1-\frac{a_{k}}{a_{j}}\right),\label{eq:Eigenvector}
\end{equation}
which can be directly computed by induction on $i$. The inverse $\mathbf{D}_{a}^{-1}=\mathbf{E}_{a}$
is equivalent to
\begin{equation}
\delta_{ij}=\sum_{t=j}^{i}\frac{a_{i}}{a_{t}}\left(\prod_{k=i+1}^{N}\left(1-\frac{a_{k}}{a_{t}}\right)\right)\cdot\left(\overset{N}{\underset{\genfrac{}{}{0pt}{}{k=j}{k\neq t}}{\prod}}\frac{1}{1-\frac{a_{k}}{a_{t}}}\right),\label{eq:InverseEigenvectorMatrix}
\end{equation}
which reduces to
\begin{equation}
\sum_{t=j}^{i}\left(\overset{i}{\underset{\genfrac{}{}{0pt}{}{k=j}{k\neq t}}{\prod}}\frac{1}{a_{t}-a_{k}}\right)=\delta_{i,j}.
\label{eq:PFD}
\end{equation}
Now, consider the partial fraction decomposition \cite[eq.~1, pp.~313]{Zeng} that
\[
\frac{1}{\left(1-a_{j}z\right)\cdots\left(1-a_{i}z\right)}=\sum_{t=j}^{i}\frac{1}{1-a_{t}z}\left(\prod_{\genfrac{}{}{0pt}{}{k=j}{k\neq t}}^{i}\frac{a_{t}}{a_{t}-a_{l}}\right).
\]
By multiplying both sides
by $z$ and then letting $z\rightarrow\infty$, we obtain \eqref{eq:PFD}. Thus, the proof is complete. 
\end{proof}
\begin{rem}
The diagonalization leads to an easy computation of powers of $\mathbf{S}_a$, i.e., 
\begin{equation}
\left(\mathbf{S}_{a}\right)^{k}=\mathbf{D}_{a}\diag\left(a_{1}^{k},\ldots,a_{N}^{k}\right)\mathbf{E}_{a}.\label{eq:Powers}
\end{equation}
\end{rem}

\section{\label{sec:Applications}Applications}

\subsection{\label{subsec:RW}Random walks }

In this subsection, we let $f_{l}\left(x\right)\equiv H_{a}\left(x\right):=1/x^{a}$ for $l=1,\ldots,k$,
where $a\ge1$. Assume $a=1$, then we have

\begin{equation}
\mathbf{S}_{H_{1}}=\left(\begin{array}{ccccc}
1 & 0 & 0 & \cdots & 0\\
\frac{1}{2} & \frac{1}{2} & 0 & \cdots & 0\\
\vdots & \vdots & \vdots & \ddots & \vdots\\
\frac{1}{N} & \frac{1}{N} & \frac{1}{N} & \cdots & \frac{1}{N}
\end{array}\right).\label{eq:SN1}
\end{equation}
Now, label $N$ sites as follows:
\[
\underset{1}{\bullet}\ \underset{2}{\bullet}\ \underset{3}{\bullet}\ \ \cdots\ \underset{N-1}{\bullet}\ \underset{N}{\bullet}
\]
and consider a random walk starting from site ``$N$'', with the
rules:
\begin{itemize}
\item one can only jump to sites that are NOT to the right of the current
site, with equal probabilities;
\item steps are independent.
\end{itemize}
Let $\mathbb{P}\left(i\rightarrow j\right)$ denote the probability
from site ``$i$'' to site ``$j$''. For example, suppose we are
at site ``$6$'':
\[
\underset{1}{\bullet}\ \underset{2}{\bullet}\ \underset{3}{\bullet}\ \underset{4}{\bullet}\ \underset{5}{\bullet}\underset{6}{\overset{\text{here}}{\bullet}}\ \underset{7}{\bullet}\ \underset{8}{\bullet}\ \ldots\ \underset{N}{\bullet}
\]
then, the next step only allows to walk to sites $\left\{ 1,2,3,4,5,6\right\} $,
with probabilities:
\[
\mathbb{P}\left(6\rightarrow6\right)=\cdots=\mathbb{P}\left(6\rightarrow1\right)=\frac{1}{6}.
\]
Therefore, a typical walk is as follows:\medskip{}

\uline{STEP $1$}: walk from ``$N$'' to some site ``$n_{1}\left(\leq N\right)$'',
with $\mathbb{P}\left(N\rightarrow n_{1}\right)=\frac{1}{N}$;\smallskip{}

\uline{STEP $2$}: walk from ``$n_{1}$'' to ``$n_{2}\left(\leq n_{1}\right)$'',
with $\mathbb{P}\left(n_{1}\rightarrow n_{2}\right)=\frac{1}{n_{1}}$;\texttt{\vspace{-25bp}
}

{\LARGE{}
\[
\ \ \ \cdots\ \ \ \cdots\ \ \ \cdots\ \ \ \cdots\ \ \ \cdots\ \ \ \cdots\ \ \ \ \ \ \ \ \ \ 
\]
}\texttt{\vspace{-15bp}
}

\uline{STEP $k+1$}: walk $n_{k}\mapsto n_{k+1}\left(\leq n_{k}\right)$,
with $\mathbb{P}\left(n_{k}\rightarrow n_{k+1}\right)=\frac{1}{n_{k}}$.\medskip{}
\\
We consider the event that after $k+1$ steps, we arrive the final
destination site ``$1$'', i.e., $\mathbb{P}\left(n_{k+1}=1\right)$.
Note that site ``$1$'' is a \emph{sink}: once you reach it, you never get out. Since the steps are independent, 
\begin{equation}
\mathbb{P}\left(n_{k+1}=1\right)=\sum_{N\geq n_{1}\geq\cdots\geq n_{k}\geq1}\frac{1}{Nn_{1}\cdots n_{k}}=\frac{1}{N}S\left(\underset{k}{\underbrace{H_{1},\ldots,H_{1}}};N,1\right).\label{eq:RandomWalk41}
\end{equation}
Meanwhile, the stochastic transition matrix is exactly given by $\mathbf{S}_{H_{1}}$,
i.e., $\mathbf{S}_{H_{1}}=\left(\mathbb{P}\left(i\rightarrow j\right)\right)_{N\times N}$.
Thus,
\begin{equation}
\left(\left(\mathbf{S}_{H_{1}}\right)^{k+1}\right)_{N,1}=\mathbb{P}\left(n_{k+1}=1\right)=\frac{1}{N}S\left(\underset{k}{\underbrace{H_{1},\ldots,H_{1}}};N,1\right),\label{eq:RW41}
\end{equation}
which is a probabilistic interpretation of (\ref{eq:MatrixS}) with the
slight difference that $\mathbf{P}$ in (\ref{eq:MatrixS}) is replaced
by $\mathbf{S}_{H_{1}}$. 
\begin{rem}
When $a>1$, we could also form a similar random walk by 
\begin{itemize}
\item adding another sink ``$R$'', to the right of ``$N$'',
with $\mathbb{P}\left(R\rightarrow n\right)=\delta_{R,n}$
for $n\in\left\{ 1,\ldots,N,R\right\} $;
\item defining for $l=1,2,\ldots,N$, 
\[
\mathbb{P}\left(l\rightarrow j\right)=\begin{cases}
0, & \text{if }l<j\leq N;\\
\frac{1}{l^{a}}, & \text{if }1\leq j\leq l;\\
1-\frac{1}{l^{a-1}}, & \text{if }j=R.
\end{cases}
\]
\end{itemize}
Now for the stochastic transition matrix,
\[
\left(\begin{matrix}\mathbf{S}_{a} & *\\ * & 1\end{matrix}\right)
\Rightarrow
\left(\begin{matrix}\mathbf{S}_{a} & *\\ * & 1\end{matrix}\right)^{k+1}
=
\left(\begin{matrix}\mathbf{S}_{a}^{k+1} & *\\ * & 1\end{matrix}\right)
\]
A similar calculation for $\mathbb{P}\left(n_{k+1}=1\right)$ shows
that 
\begin{equation}
S\left(\underset{k}{\underbrace{H_{a},\ldots,H_{a}}};N,1\right)=N^{a}\left(\left(\mathbf{S}_{H_{a}}\right)^{k+1}\right)_{N,1}.\label{eq:RW4a}
\end{equation}
\end{rem}

\subsection{\label{subsec:SANDA}Relations between $S$ and $A$}
Hoffman  \cite[pp.~275--276]{Hoffman} studied the relations between the harmonic series $S\left(1/x^{i_1},\ldots,1/x^{i_k};\infty,1\right)$ and $A\left(1/x^{i_1},\ldots,1/x^{i_k};\infty,1\right)$. For example, \cite[pp.~276]{Hoffman} 
\begin{equation}
S\left(\frac{1}{x^{i_1}},\frac{1}{x^{i_2}};\infty,1\right)=A\left(\frac{1}{x^{i_1}},\frac{1}{x^{i_2}};\infty,1\right)+A\left(\frac{1}{x^{i_1+i_2}};\infty,1\right),\label{eq:SATwoZeta}
\end{equation}
\begin{align}
\nonumber S\left(\frac{1}{x^{i_1}},\frac{1}{x^{i_2}},\frac{1}{x^{i_3}};\infty,1\right)=&A\left(\frac{1}{x^{i_1}},\frac{1}{x^{i_2}},\frac{1}{x^{i_3}};\infty,1\right)+A\left(\frac{1}{x^{i_1+i_2}},\frac{1}{x^{i_3}};\infty,1\right)\\
&+A\left(\frac{1}{x^{i_1}},\frac{1}{x^{i_2+i_3}};\infty,1\right)+A\left(\frac{1}{x^{i_1+i_2+i_3}};\infty,1\right).\label{eq:SAThreeZeta}
\end{align}
Next, we will establish the truncated and generalized versions of (\ref{eq:SATwoZeta})
and (\ref{eq:SAThreeZeta}), in the sense that we truncate the series
(from both above and below) into sums, which at the same time allows
flexibility for general summands, not restricted to negative
powers.
\begin{thm}
For positive integers $N$ and $m$ with $N>m$, we have
\begin{equation}
S\left(f,g;N-1,m\right)=A\left(f,g;N,m\right)+A\left(fg;N,m\right)\label{eq:SATwo}
\end{equation}
and 
\begin{equation}
\begin{aligned}
S\left(f,g,h;N-1,m\right)=&A\left(f,g,h;N,m\right)+A\left(fg,h;N,m\right)\\
&+A\left(f,gh;N,m\right)+A\left(fgh;N,m\right).\label{eq:SAThree}
\end{aligned}
\end{equation}
\end{thm}
\begin{proof}
By Theorem \ref{Thm:Main}, the right-hand side of (\ref{eq:SATwo})
is given by 
\[
\left(\mathbf{P}\mathbf{A}_{f}\mathbf{A}_{g}\right)_{N,m}+\left(\mathbf{P}\mathbf{A}_{fg}\right)_{N,m}=\left(\mathbf{P}\mathbf{\Delta}\left(\mathbf{S}_{f}\mathbf{\Delta}\mathbf{S}_{g}+\mathbf{S}_{fg}\right)\right)_{N,m}.
\]
From (\ref{eq:SATwoMatrix}), we see that
\[
\mathbf{I}-\mathbf{\Delta}=\left(\mathbf{S}_{f}\right)^{-1}\mathbf{S}_{fg}\left(\mathbf{S}_{g}\right)^{-1}\Leftrightarrow\mathbf{S}_{f}\mathbf{\Delta}\mathbf{S}_{g}+\mathbf{S}_{fg}=\mathbf{S}_{f}\mathbf{S}_{g}.
\]
An easy observation, noticing the different dimensions
of matrices, shows that
\[
\left(\left(\mathbf{P}_{N}\mathbf{\Delta}_{N}\right)\left(\mathbf{S}_{N\mid f}\mathbf{S}_{N\mid g}\right)\right)_{N,m}=\left(\mathbf{P}_{N-1}\mathbf{S}_{N-1\mid f}\mathbf{S}_{N-1\mid g}\right)_{N-1,m}=S\left(f,g;N-1,m\right).
\]
Similarly, (\ref{eq:SAThree}) is equivalent to (\ref{eq:SAThreeMatrix}),
due to replacement $\left(a,b,c\right)\mapsto\left(f,g,h\right)$.
\end{proof}

\subsection{\label{subsec:COMB}Combinatorial identities}

The matrix computations in Section \ref{sec:GMR}, especially the diagonalization
for computing a matrix power, lead to alternative proofs for some combinatorial
identities and their generalizations.
\begin{example}
Butler and Karasik \cite[Thm.~4, pp.~7]{Nested} showed that if $G\left(n,k\right)$
satisfies $G\left(n,n\right)=1$, $G\left(n,-k\right)=0$ and for
$k\geq1$, $G\left(n,k\right)=G\left(n-1,k-1\right)+a_{k}G\left(n-1,k\right)$, 
then 
\[
S\left(\underset{k}{\underbrace{a,\ldots,a}},N,1\right):=\sum_{N\geq n_{1}\geq\cdots\geq n_{k}\geq1}a_{n_{1}}\cdots a_{n_{k}}=G\left(N+k,N\right),
\]
based on a proof related to Stirling numbers of the second kind. In
fact through index matrices, we could provide a direct proof without
using Stirling numbers. 

\noindent 1. When $k=1$, an induction on $N$ shows directly that 
\[
\sum_{N\geq n_{1}\geq1}a_{n_{1}}=a_{N}+G\left(N,N-1\right)=a_{N}G\left(N,N\right)+G\left(N,N-1\right)=G\left(N+1,N\right).
\]

\noindent 2. For the inductive step in $k$, similarly to (\ref{eq:RW4a}), we see, by recurrence,
\begin{align*}
S\left(\underset{k}{\underbrace{a,\ldots,a}},N,1\right) & =a_{N}\left(\prod_{l=1}^{k}\mathbf{S}_{a}\right)_{N,1}=a_{N}\left(\mathbf{S}_{a}\left(\prod_{l=1}^{k-1}\mathbf{S}_{a}\right)\right)_{N,1}\\
 & =\frac{1}{a_{N}}\sum_{m=1}^{N}a_{N}\cdot a_{m}G\left(m+k-1,m\right)\\
 & =G\left(N+k,N\right).
\end{align*}
\end{example}
\begin{example}
Suppose the $\left(a_{m}\right)_{m=1}^{N}$ are all distinct. An alternative
expression of the previous example can be obtained by diagonalization.
\begin{eqnarray*}
S\left(\underset{k}{\underbrace{a,\ldots,a}},N,1\right) & = & \frac{1}{a_{N}}\left(\mathbf{D}_{H_{a}}\diag\left\{ a_{1}^{k+1},\ldots,a_{N}^{k+1}\right\} \mathbf{E}_{H_{a}}\right)_{N,1}\\
 & = & \frac{1}{a_{N}}\sum_{j=1}^{N}a_{j}^{k+1}\left(\frac{a_{N}}{a_{j}}\prod_{\genfrac{}{}{0pt}{}{m=1}{m\neq j}}^{N}\frac{1}{1-\frac{a_{m}}{a_{j}}}\right)\\
 & = & \sum_{j=1}^{N}\left(\prod_{\genfrac{}{}{0pt}{}{m=1}{m\neq j}}^{N}\frac{1}{1-\frac{a_{m}}{a_{j}}}\right)a_{j}^{k}.
\end{eqnarray*}
This recovers a general result \cite[eq.~2, pp.~313]{Zeng}, which,
when we take $a_{j}=\frac{a-bq^{j+i-1}}{c-zq^{j+i-1}}$ and $N=n-i+1$,
``turns out to be a common source of several $q$-identities'' \cite[pp.~314]{Zeng}.
The special case $a_{m}=m^{a}$ yields
\begin{equation}
S_{\underset{k}{\underbrace{a,\ldots,a}}}\left(N\right)=\sum_{l=1}^{N}\left(\prod_{\genfrac{}{}{0pt}{}{n=1}{n\neq l}}^{N}\frac{n^{a}}{n^{a}-l^{a}}\right)\frac{1}{l^{ak}},\label{eq:GeneralDilcher}
\end{equation}
which gives (\ref{eq:Dilcher}) when $a=1$.
\end{example}
\begin{rem}
When $a=m\in\mathbb{Z}_{+}$, consider the factorization 
\[
n^{m}-l^{m}=\left(n-l\right)\left(n-\xi_{m}l\right)\cdots\left(n-\xi_{m}^{m-1}l\right),
\]
where $\xi_{m}:=\exp\left\{ \frac{2\pi i}{m}\right\} $,
and $i^{2}=-1$. 
We could obtain the following binomial-type expression 
\[
S_{\underset{k}{\underbrace{a,\ldots,a}}}\left(N\right)=\sum_{l=1}^{N}\left(\prod_{t=0}^{m-1}{N \choose \xi_{m}^{t}l}\frac{\pi\left(1-\xi_{m}^{t}\right)l}{\sin\left(\pi\xi_{m}^{t}l\right)}\right)\frac{1}{l^{mk}},
\]
which is similar to \eqref{eq:Dilcher} and the usual binomial coefficient is generalized as $\binom{x}{y}:=\frac{\Gamma\left(x+1\right)}{\Gamma\left(x+1\right)\Gamma\left(x-y+1\right)}$. 
\end{rem}

\section{Acknowledgment}

This work was initiated when the first author was a postdoc in Research
Institute for Symbolic Computation, Johannes Kepler University, supported
by SFB F50 (F5006-N15 and F5009-N15) grant, and continued when the
first author switched, as a postdoc, to Johann Radon Institute for
Applied and Computational Mathematics, Austrian Academy of Science,
supported by Austrian Science Fund (FWF) grant FWF-Projekt 29467. Also, the first author would like to thank his current supervisor Dr.~Karl Dilcher, for his careful reading and valuable suggestions on this work.

\noindent The second author was supported by the National Science
Foundation of China (No.~1140149).

\end{document}